\documentclass[a4paper, 12pt]{amsart}
\usepackage{amssymb}
\usepackage{amscd}
\usepackage{eepic}
\usepackage{enumerate}
\usepackage{verbatim}   
\usepackage{xr}
\usepackage{graphicx}
\usepackage{color}
\let\le =\leqslant
\let\leq=\leqslant

\let\geq=\geqslant
\def\cinfty{C^\infty} \def\cinftyc{C^\infty _c}
\def\CC{\mathbb C}  \def\DD{\mathbb D}   \def\NN{\mathbb N}  
\def\RR{\mathbb R}     \def\ZZ{\mathbb Z}

\def\C{\mathcal C}   \def\H{\mathcal H}

\def\implies{ \Longrightarrow  } 
\def\ens#1{\{ #1 \}}
\def\Ker{\operatorname{Ker}}
\def\Pf{\operatorname{Pf}}

\def\d{\text{\,\rm d}} 
\def\ob{\operatorname{ob}} 
\def\tildeob{\operatorname{\widetilde{ob}}}
\def\pr{\operatorname{pr}}
\def\supp{\operatorname{supp}}
\def\Pf{\operatorname{Pf}} 

\def\Res{\operatorname{Res}}

\def\zeroset{\{ 0 \}}
\def\eps{\varepsilon}
\def\ph{\varphi} 
\def\rond{{\circ}}
\theoremstyle{plain}
	\newtheorem{Thm}{Theorem}[section]
	\newtheorem{Cor}[Thm]{Corollary}
	\newtheorem{Prop}[Thm]{Proposition}
	\newtheorem{Lem}[Thm]{Lemma}
	 
\theoremstyle{definition}
    \newtheorem{Def}[Thm]{Definition}
    \newtheorem{Rem}[Thm]{Remark}
    
    \newtheorem{Ex}[Thm]{Example}

\numberwithin{equation}{section}

\begin{document}

\title{Oblique poles of \ $\int_X\vert {f}\vert ^{2\lambda}\vert {g}\vert^{2\mu}\ \square$} 

\author{D. Barlet}
\address{Institut \'Elie Cartan\\
		    Nancy-Universit\'e \\
		    BP 239 \\
		    F-54506 Vandoeuvre-l\`es-Nancy}
\email{Daniel.Barlet@iecn.u-nancy.fr}	
	    
\author{H.-M. Maire}
\address{Universit\'e de Gen\`eve\\
		2-4, rue du Li\`evre \\
		 Case postale 64\\
		    CH-1211 Gen\`eve 4}
\email{henri.maire@unige.ch}

\date{\today}


\maketitle

\begin{abstract}
Existence of oblique polar lines for the meromorphic extension of the current valued function
$\int |f|^{2\lambda}|g|^{2\mu}\square$ is given under the following hypotheses:
$f$ and $g$ are holomorphic function germs in $\CC^{n+1}$ such that $g$ is non-singular,
the germ $S:=\ens{\d f\wedge \d g =0}$ is one dimensional, and $g|_S$ is proper and finite.
The main tools we use are interaction of strata for $f$ (see \cite{B:91}), monodromy of the local system $H^{n-1}(u)$ on $S$ for a given eigenvalue $\exp(-2i\pi u)$ of the monodromy of $f$,
and the monodromy of the cover $g|_S$.
Two non-trivial examples are completely worked out.
\end{abstract}
\section*{Introduction}

In the study of a holomorphic function $f$ defined in an open  neighbourhood of $0\in \CC^{n+1}$ with one dimensional critical locus $S$ started in \cite{B:91} and completed in \cite{B:08}, the main tool was to restrict $f$ to hyperplane sections  transverse to $S^*:=S\setminus \ens{0}$  and examine, for a given eigenvalue $\exp(-2i\pi u)$ of the monodromy of $f$, the local system $H^{n-1}(u)$ on $S^*$ formed by the corresponding  spectral  subspaces. Higher order poles of the current valued meromorphic function
$\int |f|^{2\lambda}\square$ at $-u-m$, some $m\in\NN$, are detected by non-extendable sections of $H^{n-1}(u)$ to $S$. 
An important part of this local system remained unexplored in \cite{B:91} and  \cite{B:08} because only the eigenvalue $1$ of the monodromy $\Theta$ of the local system  $H^{n-1}(u)$ has been considered, via the spaces \ $H^0(S^*, H^{n-1}(u))$ \ and \ $H^1(S^*,H^{n-1}(u))$.
 
In this paper, we will focus on the other eigenvalues of $\Theta$. Let us introduce an auxiliary function $t$ with the following properties:
\begin{enumerate}
\item the function  $t$ is non-singular near $0$; 
\item  the set $\Sigma : = \{ \d f \wedge \d t = 0 \} $ is a curve ;\item the restriction $t|_S : S \to \DD $ is proper and finite ; 
\item  $t|_S^{-1}(0) = \{ 0 \}$  and   $t|_{S^*}$ is a finite cover of $\DD^*: = \DD \setminus \{0\}$.
\end{enumerate}
Remark that condition (4) may always be acheived by localization near  $0$  when  conditions  (1), (2) and (3) are satisfied. These conditions are satisfied in a neighbourhood of the origin if  $(f,t)$ forms an isolated complete intersection singularity (icis) with one dimensional critical locus. But we allow also the case where  $\Sigma$  has branches in  $\{f = 0 \}$  not contained in  $S$.

The direct image of the constructible sheaf   $H^{n-1}(u)$   supported in  $S$  by $t$ will be denoted by $\H$; it is a local
system on $\DD^*$. Let $\H_0$ be the fibre of $\H$ at $t_0\in \DD^*$ and $\Theta_0$ its monodromy which is an automorphism of $\H_0$. In case where $S$ is smooth, it is possible to choose the function  $t$  in order that  $t|_S$ is an isomorphism and $\Theta_0$ may be identified with the monodromy   $\Theta$  of  $H^{n-1}(u)$  on  $S^*$. In general, $\Theta_0$ combines $\Theta$ and the monodromy of the cover  $t|_{S^*}$.\\
If  $S =: \cup_{i \in I} S_i$  is the decomposition of  $S$  into irreducible branches, we have an analoguous decomposition on  $S^*$  of the local system  $H^{n-1}(u) = \oplus_i  H^{n-1}(u)^i$  with  $\Theta = \oplus_i \ \Theta^i$. 

\smallskip 
Take an eigenvalue $\exp(-2i\pi l/k)\neq 1$ of $\Theta$, with $l\in [1,k-1]$ and $(l,k)=1$.
We define an anologue of the interaction of strata in this new context. The auxiliary non singular function  $t$  is used to realize analytically the rank one local system on  $S^*$  with  monodromy   $\exp(-2i\pi l/k)$. To perform this we need  the degree of  $t$  at the origin on the irreducible branch   $S_i$  we are interested in to be relatively prime to  $k$. Of course this is the case when  $S$  is smooth and  $t$  transversal to  $S$  at the origin. Using then a  $k-$th root of  $t$  we can lift our situation  to the case where we consider the invariant section of the  complex of  vanishing cycles of the lifted  function $\tilde{f}$ (see Theorem \ref{Thm:interaction}) and then use  already known results from \cite{B:91}.

\bigskip

With the help of elementary properties of meromorphic functions of two variables detailed in paragraph \ref{sect:polar oblique}, we  deduce from interaction of strata above the existence of oblique polar lines for the meromorphic extension of $\int_X |f|^{2\lambda}|t|^{2\mu}\square$. This result is new and consists in a first step toward the comprehension of the polar structure of such an extension.

\section{Polar structure of \ $\int_X\vert f\vert ^{2\lambda}\ \square$} 

\begin{Thm} \label{Thm:Bernstein} 
\textsc{Bernstein \& Gelfand.} For $m$ and $p \in \NN^*$,
let $Y$ be an open subset in $\CC^m$, $f:Y\to \CC^p$  a holomorphic map and
$X$ a relatively compact open set in $Y$.
Then there exists a finite set $P(f) \subset \NN^p$ such that, for any form
$\phi \in\Lambda^{m,m}\cinfty_c(X)$ with compact support, the holomorphic map
in $\ens{\Re \lambda_1>0}\times \dots \times \ens{\Re \lambda_p>0}$ given by
\begin{equation} \label{E:mero}
        (\lambda_1, \dots,\lambda_p) \mapsto\int_X |f_1|^{2\lambda_1} \dots  |f_p|^{2\lambda_p} \phi
\end{equation}
has a meromorphic extension to $\CC^p$ with poles contained in the set
$$ \bigcup_{ a\in P(f), l\in \NN^*}  \{\langle a\mid\lambda\rangle +l=0\}.
$$

\end{Thm}

\begin{proof}
For sake of completeness we recall the arguments of \cite{BG:69}.

Using desingularization of the product $f_1\dots f_p$, we know \cite{Hironaka:64}
that there exists a holomorphic manifold $\tilde Y$ of dimension $m$
and a holomorphic proper map $\pi:\tilde Y \to Y$ such that the composite functions
$\tilde f_j:= f_j\rond \pi$ are locally expressible as
\begin{equation} \label{E:normalcr}
  \tilde f_k(y) = y_1^{a_1^k}\dots y_m^{a_m^k} u_k(y),  1\le k \le p,
\end{equation}
where $a_j^k \in \NN$ and $u_k$ is a holomorphic nowhere vanishing function.
Because $\pi^{-1}(X)$ is relatively compact, it may be covered by a finite number of
open set where (\ref{E:normalcr}) is valid.

For $\ph \in\Lambda^{m,m}\cinfty_c(X)$ and $\Re \lambda_1,\dots, \Re \lambda_p$
positive, we have
$$ \int_X |f_1|^{2\lambda_1} \dots  |f_p|^{2\lambda_p} \phi  =   \int_{\pi^{-1}(X)} |\tilde f_1|^{2\lambda_1} \dots  |\tilde f_p|^{2\lambda_p} \pi^*\phi .
$$
Using  partition of unity and setting $\mu_k:=a_k^1\lambda_1+\dots+a_k^p\lambda_p$, 
$1\le k\le m$,
we are reduced to give a meromorphic extension to
\begin{equation} \label{E:mero1}
          (\mu_1,\dots,\mu_m) \to \int_{\CC^m} |y_1|^{2\mu_1} \dots |y_m|^{2\mu_m}                        \omega(\mu,y)      ,
\end{equation}
where $\omega$ is a $\cinfty $ form of type $(m,m)$ with compact support in $\CC^m$ valued in the space of entire functions on $\CC^m$. Of course, (\ref{E:mero1}) is holomorphic
in $\ens{\Re \mu_1>-1, \dots,\Re \mu_m>-1}$.

The relation
$$ (\mu_1+1).|y_1|^{2\mu_1} = \partial_1( |y_1|^{2\mu_1}.y_1)$$
implies by partial integration in $y_1$
$$  \int_{\CC^m} |y_1|^{2\mu_1} \dots |y_m|^{2\mu_m}\omega(\mu,y)
      = \frac{-1}{\mu_1+1} \int_{\CC^m} 
      |y_1|^{2\mu_1}.y_1. |y_2|^{2\mu_2 }\dots |y_m|^{2\mu_m }\partial_1\omega(\mu,y).
$$
Because    $\partial_1\omega$ is again a $\cinfty $ form of type $(m,m)$ with compact support in $\CC^m$ valued in the space of entire functions on $\CC^m$,
we may repeat this argument for each coordinate $y_2, \dots, y_m$ and obtain
$$ \begin{aligned}
     \int_{\CC^m} |y_1|^{2\mu_1} \dots |y_m|^{2\mu_m}\omega(\mu,y) &=   \\
       \frac{(-1)^m}{(\mu_1+1)\dots(\mu_m+1)} \int_{\CC^m} &
      |y_1|^{2\mu_1}.y_1. |y_2|^{2\mu_2 }.y_2 \dots |y_m|^{2\mu_m}.y_m.\partial_1\dots \partial_m\omega(\mu,y).
      \end{aligned}
$$      
The integral on the right hand side is holomorphic for $\Re \mu_1>-3/2, \dots,\Re \mu_m>-3/2$.
Therefore
the function (\ref{E:mero1}) is meromorphic in this domain with only possible poles in the union of the hyperplanes $\ens{\mu_1+1=0}, \dots, \ens{\mu_m+1=0}$.

Iteration of these arguments concludes the proof.
\end{proof}

\begin{Rem}
An alternate proof of Theorem \ref{Thm:Bernstein}  has been given for $p=1$ by Bernstein \cite{Bern:72}, Bj\"ork \cite{Bj:79}, Barlet-Maire \cite{BM:89}, and by Loeser \cite{Loe:89} and Sabbah \cite{Sab:87} in general.
\end{Rem}

In case where $f_1,\dots , f_p$ define an isolated complete intersection singularity (icis), Loeser and Sabbah gave morover the following information on the set $P(f)$ of the "slopes" of the polar hyperplanes in the meromorphic extension of the function (\ref{E:mero}): it is contained in the set of slopes of the discriminant $\Delta$ of $f$, which in this case is an hypersurface in $\CC^p$.
More precisely, take the $(p-1)-$skeleton of the fan associated to the Newton polyhedron of $\Delta $ at $0$ and denote by $P(\Delta)$ the set of directions associated to this 
$(p-1)-$skeleton union $\ens{(a_1,\dots, a_p)\in \NN^p\mid a_1\dots a_p=0}$. Then
$$ P(f) \subseteq P(\Delta).
$$
In particular, if the discriminant is contained in the hyperplanes of coordinates, then
there are no  polar hyperplanes with direction in $(\NN^*)^p$ .

The results of Loeser and Sabbah above have the following consequence for an icis which  is proved below  directly by elementary arguments, after we have introduced some terminology.

\begin{Def} 
Let  $f_1, \dots, f_p$ be holomorphic functions on an open neighbourhood   $X$  of the origin in  $\mathbb{C}^m$. We shall say that a polar hyperplane  $H \subset \mathbb{C}^p$  for the meromorphic extension of  $\int_X \vert f_1\vert^{2\lambda_1}\dots \vert f_p\vert^{2\lambda_p}\square$ \emph{is supported by the closed set  $F \subset X$}, when $H$  is not a polar hyperplane for the meromorphic extension of  $\int_{X\setminus F} \vert f_1\vert^{2\lambda_1}\dots \vert f_p\vert^{2\lambda_p}\square$.
We shall say that a polar direction is \emph{supported in} $F$ if any  polar hyperplane with this direction is supported by  $F$.
\end{Def}

\begin{Prop}
Assume $f_1,\dots, f_p$ are quasi-homogeneous for the weights  $w_1, \dots, w_p$,  of degree $a_1,\dots, a_p$.  Then if there exists a polar hyperplane direction  supported by the origin  for $(\ref{E:mero})$ in $(\NN^*)^p$ it  is $(a_1,\dots, a_p)$ and the corresponding poles are at most simple.

In particular, for $p=2$, and if  $(f_1, f_2)$  is an icis, all oblique poles have direction $(a_1, a_2)$.
\end{Prop}

\begin{proof}
Quasi-homogeneity gives 
$ f_k(t^{w_1} x_1, \dots, t^{w_m} x_m)=t^{a_k}f_k(x)$, $k=1, \dots, p$.

Let
$\Omega:=\sum _1^m (-1)^{j-1} w_j x_j \d x_0\wedge \dots\wedge\widehat{\d x_j}\wedge \dots\wedge\d x_m$ so that
$\d \Omega= (\sum w_j) \d x$.

>From Euler's relation,because $f_k^{\lambda_k}$ is quasi-homogeneous of degree $a_k\lambda_k$:
$$ \begin{aligned}
     \d f_k^{\lambda_k}\wedge \Omega&= a_k \lambda_ kf_k^{\lambda_k} \d x, \\ 
     \d x^\delta\wedge\Omega& = \langle w\mid\delta\rangle x^\delta\d x, \ \forall \delta\in \NN^m.
     \end{aligned}
$$

Take $\rho\in \C^\infty_c(\CC^{m})$; then, with ${\bf 1}= (1, \dots,1)\in \NN^p$ and $\eps\in \NN^m$:
$$\d (\vert f\vert^{2\lambda} x^\delta\bar x^\eps\rho\Omega\wedge\d\bar x)= (\langle a\mid \lambda\rangle +\langle w\mid \delta+ {\bf 1}\rangle )\vert f\vert^{2\lambda} x^\delta\bar x^\eps\rho\d x\wedge\d\bar x +\vert f\vert^{2\lambda} x^\delta \bar x^\eps\d\rho\wedge\Omega\wedge\d\bar x.
$$

Using Stokes' formula we get
$$  (\langle a\mid \lambda\rangle +\langle w\mid \delta+ {\bf 1}\rangle ) 
      \int |f|^{2\lambda} x^\delta\bar x^\eps\rho\d x \wedge\d \bar x = -\int |f|^{2\lambda} x^\delta\bar x^\eps\d\rho\wedge\Omega
      \wedge\d \bar x.
$$   

For $\rho=1 $ near $0$, $\d\rho=0$, near $0$. Therefore the right hand side has no poles supported by the origin. Now the conclusion comes from the Taylor expansion at  $0$  of the test function. 
\end{proof}

The question of whether a polar hyperplane is effectively present in the meromorphic
extension of the function (\ref{E:mero}) for a least one $\phi$ has been addressed in case $p=1$ under the name "contribution effective" in a sequence of papers by D. Barlet \cite{B:84a},  \cite{B:84b}, \cite{B:86} etc ...

For $p>1$ no general geometric conditions are known to produce poles with direction in 
$(\NN^*)^p$. In the following paragraphs, we examine the case $p=2$.

\section{Existence of polar oblique lines} \label{sect:polar oblique}

In this paragraph, we consider two holomorphic functions $f, g:Y\to \CC$, where $Y$ is an open subset in $\CC^m$ and fix a relatively compact open subset $X$ of $Y$. Without loss of generality, we assume $0\in X$. We study the possible oblique poles of the meromorphic extension of the current valued function
\begin{equation} \label{E:mero2}
        (\lambda, \mu) \mapsto\int_X |f|^{2\lambda} |g|^{2\mu} \square.
\end{equation}

The following elementary lemma is basic.

\begin{Lem} Let $M$ be a meromorphic function in $\CC^2$ with poles along a family of lines with directions in $\NN^2$. For $(\lambda_0,\mu_0)\in \CC^2$, assume
\begin{itemize}
\item[(i)] $\ens{\lambda=\lambda_0}$ is a polar line of order $\le k_0$ of $M$,
\item[(ii)] $\ens{\mu=\mu_0}$ is not a polar line of $M$,
\item[(iii)] $\lambda \mapsto M(\lambda,\mu_0)$ has a pole of order at least  $k_0+1$ at $\lambda_0$.
\end{itemize}
Then there exists $(a,b)\in (\NN^*)^2$ such that the function $M$ has a pole along the 
(oblique) line $\ens{a\lambda+b\mu= a\lambda_0+b\mu_0}$.
\end{Lem}
\begin{proof}
If $M$ does not have an oblique pole through $(\lambda_0,\mu_0)$, then the function           $(\lambda,\mu) \mapsto (\lambda-\lambda_0)^{k_0}M(\lambda,\mu)$ is holomorphic near $(\lambda_0,\mu_0)$.
Therefore, $\lambda \mapsto M(\lambda,\mu_0)$ has at most a pole of order $k_0$ at $\lambda_0$. Contradiction.
\end{proof}

It turns out that 
to check the first condition in the above lemma for the function (\ref{E:mero2}),  it is sufficient to examine the poles of the meromorphic extension of the current valued function
\begin{equation} \label{E:mero3}
        \lambda \mapsto\int_X |f|^{2\lambda}  \square.
\end{equation}
Such a simplification does not hold for general meromorphic functions. For example,
$$ (\lambda,\mu) \mapsto \frac{\lambda +\mu}{\lambda^2}
$$
has a double pole along $\ens{\lambda=0}$ but its restricition to $\ens{\mu=0}$ has only a simple pole at $0$.

\begin{Prop} \label{Prop:supportpole}
If the meromorphic extension of the current valued function $(\ref{E:mero3})$ has a pole of order $k$ at $\lambda_0\in\RR_-$, \emph{i.e.}, it has a principal part of the form
$$  \frac{T_k}{(\lambda-\lambda_0)^k} +\dots+ \frac{T_1}{\lambda-\lambda_0},
$$
at $\lambda_0$, then the  meromorphic extension of the function
$( \ref{E:mero2})$ has a pole of order 
\begin{equation} \label{E:orderexact}
    k_0:=\max\ens{0\le l\le k\mid  \supp T_l\not\subseteq \ens{g=0}} 
\end{equation}
along the line $\ens{\lambda=\lambda_0}$. 
\end{Prop}

\begin{proof}
As a consequence of the Bernstein identity (see \cite{Bj:79}), there exists $N\in\NN$ such that the extension of $\int_X |f|^{2\lambda}\phi$ in $\ens{\Re\lambda > \lambda_0-1}$ can be achieved for $\phi\in\Lambda^{m,m}\C^N_c(X)$. Our hypothesis implies that this function has  a pole of order $\le k$ at $\lambda_0$. Because $|g|^{2\mu} \phi$ is of class $\C^N$ for $\Re \mu$ large enough, the function
$$\lambda \mapsto \int_X |f|^{2\lambda}|g|^{2\mu} \phi
$$
has a meromorphic extension in $\ens{\Re\lambda > \lambda_0-1}$ with a pole of order $\le k$ at $\lambda_0$. We have proved that $( \ref{E:mero2})$ has a pole of order $\le k$
along the line $\ens{\lambda=\lambda_0}$.

Near $\lambda_0$, the extension of $\int_X |f|^{2\lambda} \phi$ writes
$$  \frac{\langle T_k,\phi\rangle}{(\lambda-\lambda_0)^k} +\dots+ \frac{\langle T_1,\phi\rangle}{\lambda-\lambda_0} + \dots .
$$
Hence that of $\int_X |f|^{2\lambda} |g|^{2\mu}\phi$ looks
$$  \frac{\langle T_k|g|^{2\mu},\phi\rangle}{(\lambda-\lambda_0)^k} +\dots+ 
\frac{\langle T_1|g|^{2\mu},\phi\rangle}{\lambda-\lambda_0} + \dots .
$$
If $\supp T_k\subseteq \ens{g=0}$, then the first term vanishes for $\Re \mu$ large enough, because $T_k$ is of finite order (see the beginning of the proof).
So the order of the pole along the line 
$\ens{\lambda=\lambda_0}$ is $\le k_0$.

Take $x_0\in \supp T_{k_0}$ such that $g(x_0)\ne 0$ and $V$ a neighborhood of $x_0$ in which $g$ does not vanish. From the definition of the support, there exists $\psi\in\Lambda^{m,m}\cinftyc(V)$ such that $\langle T_{k_0},\psi\rangle \ne 0$. With 
$\phi:=\psi|g|^{-2\mu}\in\Lambda^{m,m}\cinftyc(V)$, we get
$$ \langle T_{k_0}|g|^{2\mu},\phi\rangle = \langle T_{k_0},\psi\rangle \ne 0.
$$
Therefore, the extension of $( \ref{E:mero2})$ has a pole of order $k_0$ along the line 
$\ens{\lambda=\lambda_0}$.
\end{proof}

\begin{Cor}  \label{Cor:oblic2}
For $(\lambda_0,\mu_0)\in (\RR_-)^2$, assume
\begin{itemize}
\item[(i)] the  extension of the current valued function $(\ref{E:mero3})$ has a pole of order $k$ at $\lambda_0$,
\item[(ii)] $\mu_0$ is not an integer translate of a root of the Bernstein polynomial of $g$,
\item[(iii)] $\lambda \mapsto \Pf(\mu=\mu_0,\int_X |f|^{2\lambda} |g|^{2\mu} \square)$ has a pole of order
$l_0>k_0$ where $k_0$ is defined in $(\ref{E:orderexact})$ at $\lambda_0$.
\end{itemize}
Then the meromorphic extension of the current valued function $(\ref{E:mero2})$
has at least $l_0 - k_0$ oblique lines, counted with multiplicities, through $(\lambda_0,\mu_0)$.
\begin{center}

\includegraphics[scale=.4]{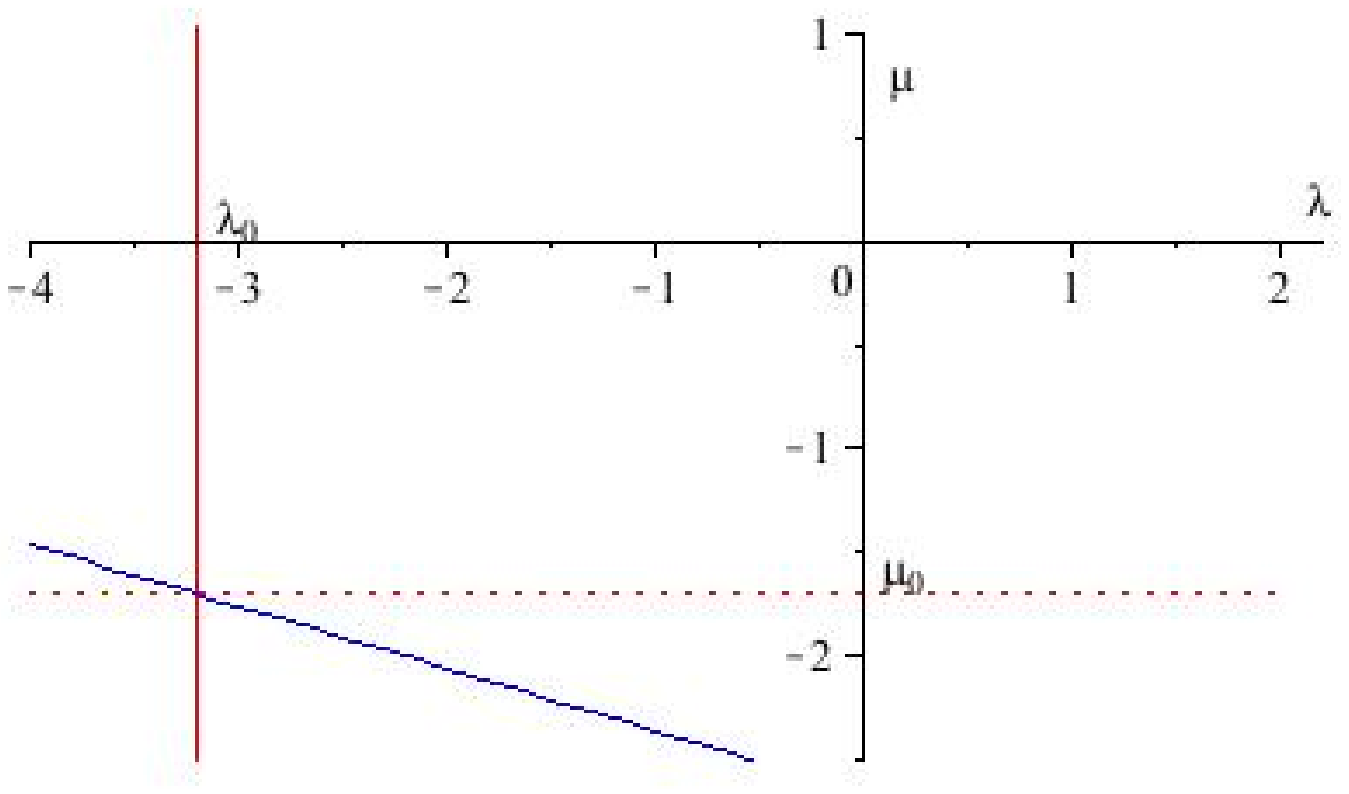}
\end{center}
\end{Cor}

\begin{Ex}
$m=3$,  $f(x,y,z)=x^2+y^2+z^2$, $g(x,y,z) = z$.
\begin{center}
\includegraphics[scale=.4]{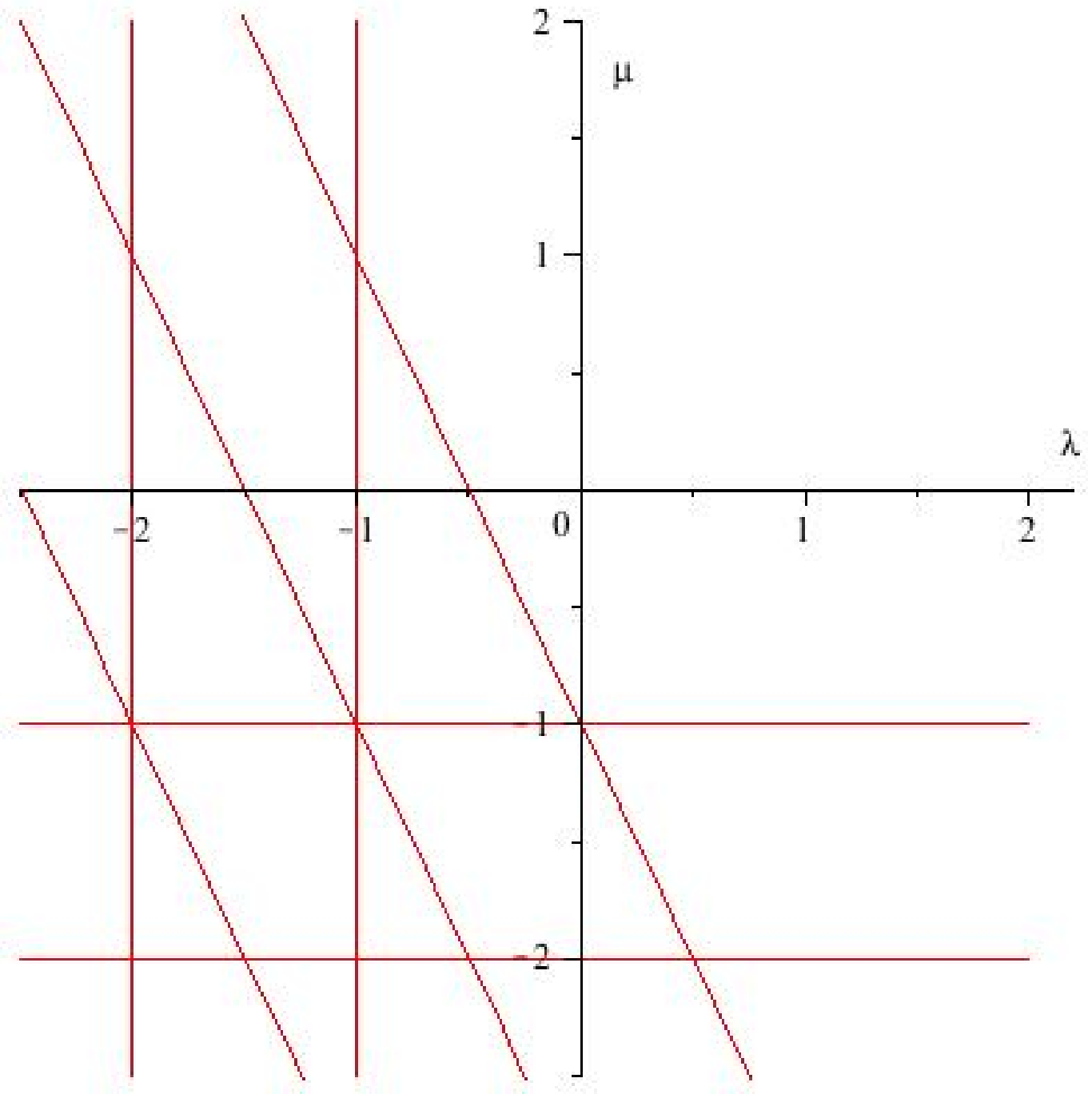}
\end{center}
\end{Ex}

\begin{Ex}
$m=4$,  $f(x,y,z)=x^2+y^2+z^2+t^2$, $g(x,y,z,t) = t^2$.
\begin{center}
\includegraphics[scale=.4]{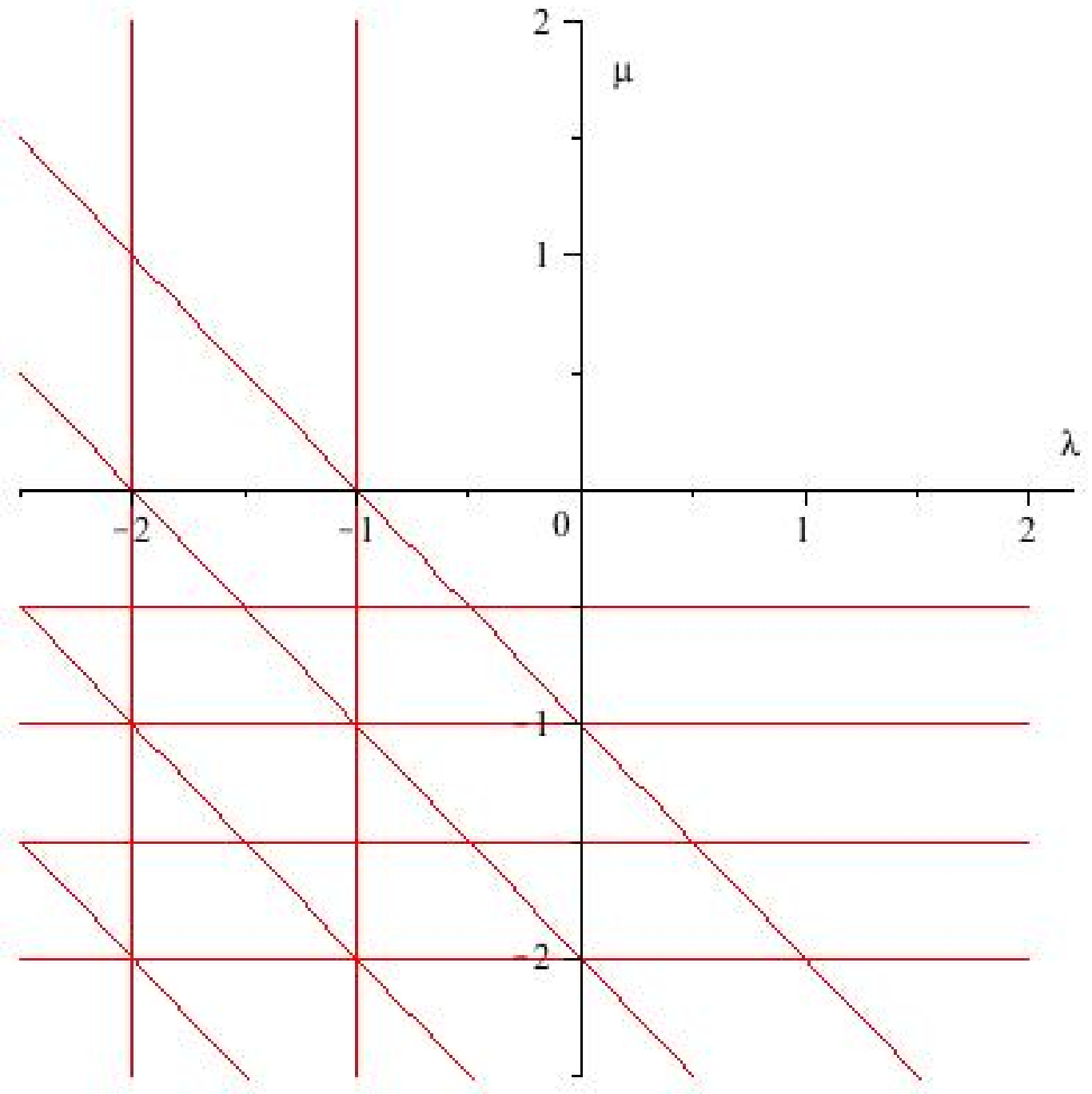}
\end{center}
\end{Ex}

\begin{Ex}
$m=3$,  $f(x,y,z)=x^2+y^2$, $g(x,y,z)=y^2+z^2$.
\begin{center}
\includegraphics[scale=.4]{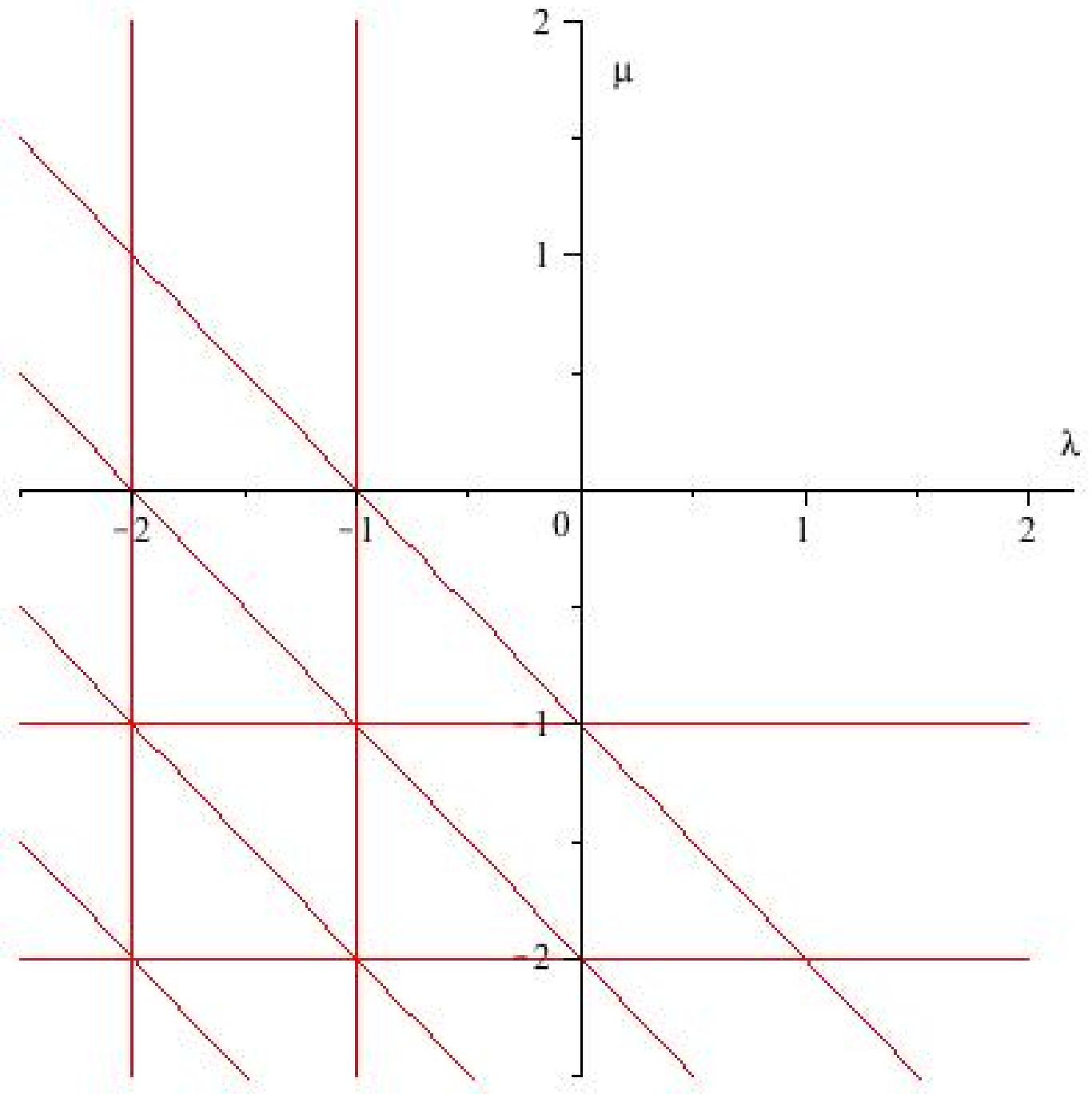}
\end{center}
In this last example, Corollary \ref{Cor:oblic2} does not apply because for $\lambda_0=-1$
we have $k_0=l_0$. Existence of an oblique polar line through $(-1,0)$ is obtained by computation of the extension of $\lambda\mapsto\Pf(\mu=1/2,\int_X |f|^{2\lambda} |g|^{2\mu} \square)$.
\end{Ex}

\section{Pullback and interaction}  \label{section:pullback}

In this paragraph, we give by pullback a method to verify condition (iii) of Corollary
\ref{Cor:oblic2} when $g$ is a coordinate. As a matter of fact the function 
$\lambda \mapsto  \int_X |f|^{2\lambda} |g|^{2\mu_0} \square$ is only known by meromorphic extension (via Bernstein identity) when $\mu_0$ is negative; it is in general difficult to exhibit some of its poles.

We consider therefore only one holomorphic function $f:Y\to \CC$, where $Y$ is an open subset in $\CC^{n+1}$ and fix a relatively compact open subset $X$ of $Y$. The coordinates in $\CC^{n+1}$ will be denoted by $x_1, \dots, x_n, t$. Let us introduce also the finite map
\begin{equation}  \label{E:revet}
p:\CC^{n+1} \to \CC^{n+1} \ \text{such that}\  p(x_1,\dots,x_n,\tau)=(x_1,\dots,x_n,\tau^k)
\end{equation}
 for some fixed integer $k$. Finally, put $\tilde f := f\rond p :\tilde X \to \CC$ where $\tilde X:=p^{-1}(X)$.

\begin{Prop} \label{Prop:pullback}
With the above notations and $\lambda_0\in\RR_-$ suppose
\begin{itemize}
\item[(a)] the  extension of the current valued function $(\ref{E:mero3})$ has a pole of order $\le 1$ at $\lambda_0$,
\item[(b)] $\lambda \mapsto \int_{\tilde X} |\tilde f|^{2\lambda} \square$ has a double pole at 
$\lambda_0$.
\end{itemize}
Then there exists  $l \in [1,k-1]$ such that the
extension of the current valued function $\lambda\mapsto  \int_{ X} | f|^{2\lambda} 
|t|^{-2l/k}\square$ has a double pole at $\lambda_0$.
\end{Prop}

\begin{proof}
Remark that  the support of the polar part of order 2 of   $ \int_{\tilde X} |\tilde f|^{2\lambda} \square$  at $\lambda_0$ is contained in $\ens{\tau=0}$ because we assume (a) and  $p$  is a local isomorphism outside  $\{\tau = 0 \}$. \\
After hypothesis (b) there exists  
$\varphi \in \Lambda^{n+1,n +1}\mathcal{C}^{\infty}_c(\tilde{X})$ such that 
$$ A : =   P_2 \big(\lambda = \lambda_0,  \int_{\tilde{X}} \vert \tilde{f }\vert ^{2\lambda}\varphi\big) \not= 0 .$$
Consider the  Taylor  expansion of $\varphi$ along $\ens{\tau=0}$ 
$$ \varphi(x,\tau) = \sum_{j+j' \leq N}  \tau^j\bar{\tau}^{j'}\varphi_{j,j'}(x)\wedge d\tau\wedge d\bar{\tau}  + o(\vert \tau \vert^N) $$
where $N$ is larger than the order of  the current  defined by $P_2$ on a compact set  $K$  containing the support of  $\varphi$.
Therefore
$$ A = \sum_{j+j' \leq N} \ P_2 \big(\lambda = \lambda_0,  \int_{\tilde{X}} \vert \tilde{f }\vert^{2\lambda}\tau^j\bar{\tau}^{j'}\varphi_{j,j'}(x)\chi(|\tau|^{2k})\wedge d\tau \wedge d\bar{\tau} \big) $$
where $\chi$ has support in \ $K$ \  and is equal to $1$ near $0$. Because $A$ does not vanish there exists   $(j,j')\in\NN^2$ with $ j + j' \leq N$ such that
$$ A_{j,j'} : =  P_2 \big(\lambda =  \lambda_0,  \int_{\tilde{X}} \vert \tilde{f }\vert ^{2\lambda}\tau^j\bar{\tau}^{j'}\varphi_{j,j'}(x)\chi(|\tau|^{2k})\wedge d\tau \wedge d\bar{\tau}\big) \not= 0 .$$ 
The change of variable  $\tau \to \exp(2i\pi/k)\tau$ that leaves $\tilde{f}$  invariant, shows that \\
$A_{j,j'} =  \exp(2i\pi(j-j')/k)A_{j,j'}$.  Hence  $A_{j,j'} = 0$ for  $j-j' \not\in k\mathbb{Z}$. 
We then get the existence of $(j,j')\in\NN^2$ verifying $j' = j + k\nu$ with  $\nu \in \mathbb{Z}$  and  $A_{j,j'} \not= 0 $.

The change of variable  $t = \tau^k$  in the computation of $A_{j,j'}$ gives
$$ P_2\big(\lambda = -\lambda_0, \int_X \vert f \vert^{2\lambda}\vert t \vert^{2(j-k+1)/k}\bar{t}^{\nu}\varphi_{j,j+k\nu}(x)\chi(|t|^2)\wedge dt\wedge d\bar{t}\big) \not= 0 .$$
This ends the proof with  $- l = j-k+1$ if $\nu \geq 0$ and with $-l = j'-k+1 $ if 
 $\nu < 0$. 
 Notice that  $l < k $ in all cases. Necessarily $l \not= 0$ because from hypothesis (a), we know that the extension of the function $(\ref{E:mero3})$  does not have a double pole  at  
$\lambda_0$. 
\end{proof}

\begin{Thm}\label{Thm:interaction}
For $Y$ open in $\CC^{n+1}$ and $X$ relatively compact open subset of $Y$,
let $f:Y\to\CC$ be holomorphic and $g(x,t)=t$. Assume $(f,g)$
satisfy properties $(1)$ to $(4)$ of the Introduction. Moreover suppose
\begin{enumerate}[(a)]
\item $\int |f|^{2\lambda}\square$ has a at most a  simple pole at $\lambda_0 - \nu$ , $ \forall \nu \in \mathbb{N}$;
\item $e^{2\pi i \lambda_0}$  is a eigenvalue of the monodromy of  $f$  acting on the $H^{n-1}$ of the Milnor fiber of  $f$  at the generic point of a connected component $S_i^*$ of $S^*$, and there exist a non zero eigenvector such its monodromy around $0$ in  $S_i$  is a primitive  $k-$root of unity, with  $k \geq 2$;
\item the degree  $d_i$  of the covering  $t : S_i^* \to \DD^* $  is prime to  $k$;
\item  $e^{2\pi i \lambda_0}$  is not an eigenvalue of the monodromy of  $\tilde{f}$  acting on the  $H^{n-1}$  of the Milnor fiber of  $\tilde{f}$   at  $0$  , where 
$\tilde f(x,\tau)=f(x_1,\dots,x_n,\tau^k)$.
\end{enumerate}

Then there exists an oblique pole of $\int |f|^{2\lambda} |g|^{2\mu}\square$ through
$(\lambda_0-j,-	l/k)$, some $j\in \NN$, and some   $l \in [1,k-1]$.
\end{Thm} 

Remark that condition (b) implies that  $\lambda_0 \not\in \mathbb{Z}$  because of the result of \cite{B:84b}.

\begin{proof}
Notice first that $(\lambda,\mu) \mapsto \int |f|^{2\lambda} |g|^{2\mu}\square$ has a simple
pole along $\ens{\lambda=\lambda_0}$. Indeed the support of the residue current of
$\lambda \mapsto \int | f|^{2\lambda} \square$ at $\lambda_0$ contains $S_i^*$ where $t$ does not vanish and Proposition \ref{Prop:supportpole} applies.

 Denote by $z$ a local coordinate on the normalization of $S_i$. The function $t$ has a zero of order $d_i$ on this normalization and hence  $d_i$  is the degree of the cover $S_i^* \to \DD^*$ induced by $t$. Without loss of generality, we may suppose $t=z^{d_i}$ on the normalization of $S_i$.

\begin{Lem}\label{revetement}
Let $k, d \in \mathbb{N}^*$ and put
$$ Y_d: = \{ (z,\tau) \in \DD^2 /  \tau^k = z^d \} ,  Y_d^*:=Y_d\setminus \zeroset.$$
If $k$ and $d$ are relatively prime, then the first projection $ \pr_{1,d}:Y_d^*\to \DD^*$ is a cyclic cover  of degree $k$.
\end{Lem}
\begin{proof}
Let us prove that the cover defined by $\pr_{1,d}$ is isomorphic to the cover 
defined by $\pr_{1,1}$, that may be taken as definition of a cyclic cover of degree $k$.

After B\'ezout's identity, there exist $a, b\in \ZZ$ such that 
\begin{equation} \label{E:Bezout}
	 ak +bd=1.
\end{equation}
Define $\ph:Y_d\to \CC^2$ by $\ph(z,\tau)=(z,z^a\tau^b)$. From (\ref{E:Bezout}), we have
$\ph(Y_d)\subseteq Y_1$ and clearly $\pr_{1,1}\rond\ph = \pr_{1,d}$.

The map $\ph$ is injective because 
$$ \ph(z,\tau) = \ph(z,\tau') \implies \tau^k=\tau'^k \ \text{and} \ \tau^b=\tau'^b,
$$
hence $\tau = \tau'$, after (\ref{E:Bezout}). It is also surjective: take $(z,\sigma)\in Y_1^*$;
the system 
$$ \tau^b=\sigma z^{-a}, \ \tau^k=z^d, \ \text{when} \  \sigma^k=z,
$$
has a unique solution because the compatibility condition $\sigma^k z^{-ak} = z^{bd}$
is satisfied.
\end{proof}

{\it End of proof of Theorem \ref{Thm:interaction}.}
Take the eigenvector with monodromy   $\exp(-2i\pi l/k)$ \ with \ $(l,k) = 1$ given by condition (b). Its   pullback by $p$ becomes invariant under the monodromy of $\tau$ because of the condition (c) and the lemma given above. After (d), this section does not extend through $0$. So we have interaction of strata (see \cite{B:91}) and a double pole for $\lambda\mapsto\int |\tilde f|^{2\lambda}\square$ at $\lambda_0-j$ with some $j\in\NN$. It remains to use Proposition \ref{Prop:pullback} and Corollary \ref{Cor:oblic2}.
\end{proof}

\section{Interaction of strata revised}

In this paragraph notations and hypotheses are those of the Introduction. Here, the function $t$ is the last coordinate.
We suppose that the eigenvalue  $\exp(-2i\pi u)$ of the monodromy of  $f$ is simple at each point of $S^*$. Therefore, this eigenvalue is also simple for the  monodromy acting on the group $H^{n-1}$ of the Milnor fibre of  $f$ at $0$. In order to compute the constructible sheaf
$H^{n-1}(u)$ on $S$ we may use the complex  
$(\Omega_X[f^{-1}], \delta_u)$, that is the complex of meromorphic forms with poles in
 $f^{-1}(0) $ equipped with the differential $\delta_u : = d  - u\frac{\d f}{f}\wedge $ along $S$. This corresponds to the case $k_0 = 1$  in \cite{B:91}. 
 
 We use the isomorphisms
\begin{equation}  \label{E:isoms}
  r^{n-1} : h^{n-1} \to  H^{n-1}(u) \ \text{over} \  S \quad \text{and} \quad 
  \tau_1 : h^{n-1} \to h^n   \ \text{over} \    S^*,
\end{equation}  
where $h^{n-1}$  [resp. $h^n$]  denotes the $(n-1)-$th [resp.   $n-$th]  cohomology sheaf of the complex  $(\Omega_X[f^{-1}], \delta_u)$.

\smallskip

In order to look at the eigenspace for the eigenvalue \ $\exp(-2i\pi l/k)$ \ of the monodromy \ $\Theta$ \ of the local system \ $H^{n-1}(u)$ \ on \ $S^*$, it will be convenient to consider the complex of sheaves
$$ \Gamma_l : = \big(\Omega^{\bullet}[f^{-1},t^{-1}], \delta_u - \frac{l}{k}\frac{dt}{t}\wedge \big) $$
which is locally isomorphic along \ $S^*$ \  to \  $(\Omega_X[f^{-1}], \delta_u)$ \ via the choice of a local branch of \ $t^{l/k}$ \ and the morphism of complexes  \ $(\Omega_X[f^{-1}], \delta_u) \to \Gamma_l$ \ given by  \ $\omega \mapsto t^{l/k}.\omega$ \ which satisfies 
$$\delta_u(t^{l/k}\omega) -\frac{l}{k}\frac{dt}{t}\wedge t^{l/k}\omega =  t^{l/k}\delta_u(\omega) .
$$
 But notice that this complex \ $\Gamma_l$ \ is also defined near the origin. Of course, a global section \ $ \sigma \in H^0(S^*, h^{n-1}(\Gamma_l))$ \ gives, via the above local isomorphism, a multivalued global section on \ $S^*$ \  of the local system \ $H^{n-1}(u) \simeq h^{n-1}$ \ with monodromy \ $\exp(-2i\pi l/k)$ \ (as multivalued section).\\
So a global meromorphic differential $(n-1)-$form \ $\omega$ \ with poles in \ $\{f.t = 0 \}$ \ such that \ $d\omega = u\frac{\d f}{f}\wedge\omega + \frac{l}{k}\frac{dt}{t}\wedge\omega$ \ defines such a \ $\sigma$, and an element in \ $\mathcal{H}_0$ \ with monodromy 
$\exp(-2i\pi l/k)$.\\
We shall use also the morphism of complexes of degree +1 
$$ \check{\tau}_1: \Gamma_l \to \Gamma_l $$
given by \ $\check{\tau}_1(\sigma) = \frac{df}{f}\wedge \sigma $. It is an easy consequence of \cite{B:91} 
that in our situation \ $ \check{\tau}_1$ \ induces an isomorphisme \ $\check{\tau}_1 : h^{n-1}(\Gamma_l) \to h^n(\Gamma_l)$ \ on \ $S^*$, because we have assumed that the eigenvalue \ $\exp(-2i\pi u)$ \ for the monodromy of  $f$  is simple along  $S^*$.

\bigskip

Our first objective is to build for each \ $ j\in\NN$ \ a morphism of sheaves on $S^*$
\begin{equation} \label{E:morphism}
r_j :  h^{n-1}(\Gamma_l) \to \underline{H}^n_{[S]}(\mathcal{O}_X), 
\end{equation}
via the meromorphic extension of $\int_X \vert f \vert^{2\lambda}|t|^{2\mu}\square$.
Here $\underline{H}^n_{[S]}(\mathcal{O}_X)$ denotes the subsheaf of the moderate cohomology with support $S$ of the sheaf $\underline{H}^n_{S}(\mathcal{O}_X)$. 
It is given by the  $n-$th cohomology sheaf of the  Dolbeault-Grothendieck complex with support $S$ :
$$ \underline{H}^n_{[S]}(\mathcal{O}_X)  \simeq \mathcal{H}^n(\underline{H}^0_S(\mathcal{D}b_X^{0,\bullet}), \d'') .$$
Let  $w$ be a $(n-1)-$meromorphic form with poles in \ $\{f = 0 \}$, satisfying \ $\delta_u(w) = \frac{l}{k}\frac{\d t}{t}\wedge w$ \ on an open neighbourghood \ $U \subset X \setminus \{t = 0\}$ \ of a point in \ $S^*$. Put for   $j \in \mathbb{N}$ :
$$ \overline{r_j(w)} : = \Res\big(\lambda = - u, \Pf(\mu=-l/k,\int_U \ \vert f \vert^{2\lambda}|t|^{2\mu}\bar{f}^{-j} \frac{\d f}{f} \wedge w \wedge \square) \big).$$
These formula define $\d'-$closed currents of type $(n,0)$ with support in $S^*\cap U$.

Indeed it is easy to check that the following formula holds in the sense of currents on $U $:
$$\begin{aligned}
    \d' \big[\Pf\big(&\lambda = - u, \Pf(\mu=-l/k,\int_{U} \  \vert f \vert^{2\lambda}|t|^{2\mu}\bar{f}^{-j}  w \wedge \square) \big)\big] =\\
     &\Res\big(\lambda = - u, \Pf(\mu=-l/k,\int_{U} \ \vert f \vert^{2\lambda}|t|^{2\mu}\bar{f}^{-j} \frac{\d f}{f} \wedge w \wedge \square) \big).
     \end{aligned}
$$

On the other hand, if $ w = \delta_u(v)- \frac{l}{k}\frac{\d t}{t}\wedge v$ \  for \  $v \in \Gamma(U,  \Omega^{n-2}[f^{-1}])$, then
$$ \begin{aligned}
     \d' \big[\Res\big(&\lambda = - u,  \Pf(\mu=-l/k,\int_{U} \ \vert f \vert^{2\lambda}|t^{2\mu}
     \bar{f}^{-j}   \frac{\d f}{f} \wedge v\wedge\square \big)\big] = \\ 
     &\Res\big(\lambda = - u,  \Pf(\mu=-l/k,\int_{U} \ \vert f \vert^{2\lambda}|t|^{2\mu}\bar{f}^{-j}   
     \frac{\d f}{f} \wedge w \wedge \square )\big)
     \end{aligned}
$$
 because  the meromorphic extension of \ $\int_X \vert f\vert^{2\lambda}\square $ has no double poles at  $\lambda \in -u - \mathbb{N}$  along  $S^*$, since  $\exp(-2i\pi u)$  is a simple eigenvalue of the monodromy of  $f$ along  $S^*$. It follows that the morphism of sheaves $(\ref{E:morphism})$ is well defined on  $S^*$.

By direct computation we show the following equality between sections  on  $S^*$ 
of the sheaf $ \underline{H}^n_{[S]}(\Omega_X^1)$ :
$$ \d'r_j(w) = -(u + j)\d f \wedge r_{j+1}(w) $$
where $\d' :  \underline{H}^n_{[S]}(\mathcal{O}_X) \to \underline{H}^n_{[S]}(\Omega_X^1)$ is the morphism induced by the  de Rham  differential $\d : \mathcal{O}_X \to \Omega_X^1 $.

Because  $ \underline{H}^n_{[S]}(\mathcal{O}_X)$ is a sheaf of $\mathcal{O}_X-$modules, it is possible to define the product $g. r_j$ for  $g$  holomorphic near a point of $S^*$ and the usual rule holds
 $$\d'(g .r_j) = \d g\wedge r_j + g.\d'\rho_j.$$
 
 Now we shall define, for each irreducible component $S_i$  of $S$  such that the local system $H^{n-1}(u)^i$ has  $\exp(-2i\pi l/k)$  as  eigenvalue for its monodromy  $\Theta^i$,  linear maps
 $$ \rho_j^i : \Ker \big(\Theta^i - \exp(-2i\pi l/k)\big) \to H^0(S^*_i, \underline{H}^n_{[S]}(\mathcal{O}_X)) $$
 as follows :\\
Let \ $s_i\in S^*_i$ \ be a base point and let \ $\gamma \in H^{n-1}(u)_{s_i}$  be such that $\Theta^i(\gamma) = \exp(2i\pi l/k).\gamma$. Denote by  $\sigma(\gamma)$  the multivalued section of the local system  $H^{n-1}(u)^i$ on  $S^*_i$  defined by  $\gamma$. Near each point of $s \in S^*_i$ we can induce $\sigma$ by a meromorphic $(n-1)-$form $w_0$ which is $\delta_u-$closed. Choose a local branch of $t^{1/k}$ near the point  $s$ and put $w:=t^{l/k}w_0$. Then it is easy to check that we define  in this way a global section $\Sigma(\gamma)$ on $S^*_i$ of the sheaf $h^{n-1}(\Gamma_l)$ which is independent of our choices. Now set 
 $$\rho_j^i(\gamma) : =  r_j(\Sigma(\gamma)).$$
 
\bigskip

Like in paragraph \ref{section:pullback}, define $\tilde f:\tilde X\to\CC$ by $\tilde f:=f\rond p$ with $ p$ of $(\ref{E:revet})$.
The singular locus $\tilde S$ of $\tilde f$ is again a curve, but it may have components contained in $\ens{\tau=0}$ (see for instance Example \ref{Ex:premier}). Let $\tilde{S}^* : = \tau^{-1}(\DD^*)$ 
(so in  $\tilde{S}$  we forget about the
components that are in  $\tau^{-1}(0)$) and define the local system $\tilde{\mathcal{H}}$ 
on  $\DD^*$  as \ $\tau_*(\tilde{H}^{n-1}(u)|_{\tilde{S}^*})$.
Denote its fiber
$\tilde\H_0$  at some $\tau_0$ with $\tau_0^k=t_0$ and the monodromy $\tilde\Theta_0$ of $\tilde\H_0$ .

We have
$$ \tilde\Theta_0 = (  \pi_*)^{-1}\rond\Theta_0\rond  \pi_*, \ \text{where} \
     \pi(\tau):=\tau^k,
$$
and $ \pi_*: \tilde\H_0\to \H_0$ is the isomorphism induced by $ \pi$.

\bigskip

Choose now the base points $s_i$ of the connected components $S^*_i$ of $S^*$ in $\{ t = t_0\}$ where $t_0$ is the base point of $\DD^*$. Moreover choose the base point  $\tau_0\in \CC^*$  such that  $\tau_0^k=t_0$.

In order to use the results of \cite{B:91}, we need to guarantee that  for the component
$S_i^*$ of $S^*$, the map $p^{-1}(S_i^*)\to S_i^*$ is the  cyclic cover of degree $k$.\\
 Fix a base point  $\tilde s_i\in p^{-1}(S^*_i)$ such that $p(\tilde{s}_i) = s_i$. The local system $\tilde H^{n-1}(u)$ on 
the component $\tilde S^*_i$ of $\tilde S^*$ containing $\tilde s_i$ is given by $\tilde H^{n-1}(u)_{\tilde s_i}$  which is isomorphic to $ H^{n-1}(u)_{ s_i}$,
and the monodromy automorphism  $\tilde{\Theta}^i$. In case $p:\tilde S^*_i \to S^*_i$
is the cyclic cover of degree $k$, we have $\tilde{\Theta}^i= (\Theta^i)^k$. 

After Lemma \ref{revetement}, this equality is true if $k$ is prime to the degree $d_i$ of the covering $ t : S_i^* \to \DD^*$.  

Let  $\tilde{\gamma}$ be the element in $(\tilde{H}^{n-1}(u)^i)_{\tilde{s_i}}$ whose image by $p$ is $\gamma$. Let $\sigma(\tilde{\gamma})$ the multivalued section of the local system $\tilde{H}^{n-1}(u)^i$ on $\tilde{S}^*_i$ given by $\tilde{\gamma}$ on $\tilde{S}_i^*$. By construction, if $(k,d_i)=1$, we get $\tilde{\Theta^i}\tilde{ \gamma} = \tilde{\gamma}$. Therefore $\sigma( \tilde \gamma)$ is in fact  a global  (singlevalued) section of the local system  $\tilde H^{n-1}(u)^i$ over $\tilde S^*_i$.

\begin{Thm} \label{Thm:noninter}
Notations and hypotheses are those introduced above. 
Take  $\gamma \in H^{n-1}(u)_{s_i}$ such that $\Theta^i(\gamma) = \exp(-2i\pi l/k)\gamma$
where $\Theta^i$ is the monodromy of $H^{n-1}(u)_{s_i}$ and  $l$ is an integer prime to  
$k$, between $1$ and $k-1$. \\
 Assume that $k$ is relatively prime to the degree of the cover $t|_{S_i^*}$ of $\DD^*$.\\
If the section $\sigma(\tilde\gamma)$  of  $\tilde{H}^{n-1}(u)^i$  on   $\tilde{S}^*_i$  defined by   $\gamma$   is the restriction to  $\tilde{S}^*_i$  of a global section  $W$   on  $\tilde{S}$  of the constructible sheaf  $\tilde{H}^{n-1}(u)$   then there exists
 $\omega \in \Gamma(X, \Omega^{n-1}_X)$  such that  the following properties are satisfied:
\begin{enumerate}
\item  $\displaystyle \d\omega = (m+u)\frac{\d f}{f} \wedge \omega + \frac{l}{k}\frac{\d t}{t} \wedge \omega $, for some $m\in\NN$;
\item The $(n-1)-$meromorphic $\delta_u-$closed   form $t^{-l/k}\omega/f^m$   induces a section on  $S$  of the sheaf   $h^{n-1}(\Gamma_l)$  whose restriction to  $S^*_i$  is given by  $\Sigma(\gamma)$;
\item the current
$$T_j : =  \Res\big( \lambda = -m-u,Pf(\mu = -l/k, \int_X \ \vert f\vert^{2\lambda}\bar{f}^{m-j}\vert t \vert^{2\mu}\frac{\d f}{f}\wedge \omega \wedge \square) \big) $$
satisfies $\d'T_j=\d'K_j$ for some current $K_j$ supported in the origin and  
$T_j - K_j$ is a $(n,0)-$current supported in  $S$ whose conjugate   induces a global section on  $S$  of the sheaf  $\underline{H}^n_{[S]}(\mathcal{O}_X)$  which is equal to  $r_j(\gamma)$  on  $S^*_i$.
\end{enumerate}
\end{Thm}

\begin{proof}
After \cite{B:84a} and \cite{B:91}, there exist an integer $m \geq 0$ and a $(n-1)-$holomorphic form  
$\tilde{\omega}$ on $X$ verifying the following properties:
\begin{enumerate}[(i)]
\item $\displaystyle \d\tilde{\omega} = (m+u)\frac{\d\tilde{f}}{\tilde{f}}\wedge \tilde{\omega} $;
\item along $\tilde{S}$ the meromorphic $\delta_u-$closed form $\displaystyle \frac{\tilde{\omega}}{\tilde{f}^{m}}$  induces the section $W$;
\item the current  
$$\tilde{T}_j : =  \Res\big(\lambda = - m - u, \int_{\tilde{X}} \vert \tilde{f}\vert^{2 \lambda} \bar{\tilde{f}}^{m-j} \frac{\d\tilde{f}}{\tilde{f}} \wedge \tilde{\omega} \wedge \
\square \big)$$
satisfies $\d'\tilde T_j=\d'\tilde K_j$ for some current $\tilde K_j$ supported in the origin and 
$\tilde T_j - \tilde K_j$ is a $(n,0)-$current supported in  $\tilde S$ whose conjugate   induces the element  $r_j(W)$   of $H^n_{\tilde{S}}(\tilde{X}, \mathcal{O}_{\tilde{X}})$.
\end{enumerate}

\smallskip

On $\tilde X$ we have an action of the group $\mathfrak{G}_k$ of $k-$th roots of unity that is given by
$Z.(x, \tau) : = (x, \zeta\tau)$ where  $\zeta : = \exp(2i\pi/k)$. Then $X$  identifies to the complex smooth quotient of $\tilde X$ by this action. In particular every
 $\mathfrak{G}_k-$invariant  holomorphic  form on $\tilde X$ is the pullback of a holomorphic form on $X$. For the holomorphic form $\tilde \omega \in \Gamma(\tilde{X}, \Omega_{\tilde{X}}^{n-1})$ above, we may write 
$$ \tilde \omega = \sum_{l = 0}^{k-1}  \tilde{\omega}_\ell,\quad  \ \text{with} \quad  Z^*\tilde \omega _\ell = \zeta^\ell \tilde \omega _\ell. 
$$ 
Indeed, $\tilde \omega _\ell =  \frac{1}{k}\sum_{j,\ell = 0}^{k-1} \zeta^{-j\ell} (Z^j)^*\tilde{\omega }$
does the job. Because $\tau^{k-\ell}\tilde \omega _\ell$ is $\mathfrak{G}_k-$invariant, there exist holomorphic forms $\omega _0, \dots , \omega _{k-1}$ on $X$ such that
\begin{equation} \label{E:decomp}
\tilde\omega = \sum_{\ell=0}^{k-1} \tau^{\ell-k} p^*\omega _\ell.
\end{equation}
Put $\omega := \omega_{k-l}$. 
Because property (i) above is $Z-$invariant, each $\omega_\ell$ verifies it and hence $\omega$, whose pullback by $p$ is $\tau^{k-l}\tilde\omega_{k-l}$, will satisfy the first condition of the Theorem, after the injectivity of $p^*$.

The action of $Z$ on $\tilde \gamma$ is $Z\tilde \gamma=\zeta^{-l}\tilde \gamma$;
therefore $\tilde\omega_{k-l}$ verifies (ii) above and hence for $\ell\neq k-l$, the form $\tilde\omega_\ell$ induces $0$ in $\tilde{H}^{n-1}(u)$ along $\tilde{S}$.

Let us prove property (3) of the Theorem. When $\tilde\omega$ is replaced by $\tilde\omega_{k-l}$ in the definition of $\tilde T_j$, the section it defines on $\tilde{S}$ does not change. On the other hand, the action of $Z$ on this section is given by multiplication by $\zeta^{-l}$.
Because this section extends through $0$, the same is true for $\tau^{k-l}\tilde T_j$
whose conjugate will define a $\mathfrak{G}_k-$invariant section of 
$ \underline{H}^n_{[S]}(\mathcal{O}_{\tilde X})$ extendable through $0$. Condition (3) follows from the isomorphism of the subsheaf of $\mathfrak{G}_k-$invariant sections of 
$ \underline{H}^n_{[ \tilde S]}(\mathcal{O}_{\tilde X})$ and $ \underline{H}^n_{[S]}(\mathcal{O}_{X})$.
\end{proof}

\bigskip

Our next result treats the case where there is a  section  $ \tilde{W}$  of  $\tilde{H}^{n-1}(u)$  on  $\tilde{S}^*$  which is \emph{not extendable at the origin} and induces  $\tilde{\gamma}$  on  $\tilde{S}^*_i$.
Remark that there always exists a global section on  $\tilde{S}^*$  inducing  $\tilde{\gamma}$  on  $\tilde{S}^*_i$ : just put  $0$  on the branches  $\tilde{S}^*_{i'}$  for each  $i'\not= i$. The next theorem shows that is this case we obtain an oblique pole of 
$\int_X|f|^{2\lambda}|g|^{2\mu}\square$.

\bigskip

Remark that in any case we may apply the previous theorem or the next one. When  $S^*$  is not connected, it is possible that both apply, because it may exist at the same time a global section on  $\tilde{S}^*$  of the sheaf  $\tilde{H}^{n-1}(u)$  inducing  $\tilde{\gamma}$  on  $\tilde{S}^*_i$ which is extendable at the origin, and another one which is not extendable at the origin.
 
\begin{Thm}  \label{Thm:inter}
Under the hypotheses of Theorem $\ref{Thm:noninter}$, assume that we have a global  section  $\tilde{W}$   on  $\tilde{S}^*$  of the local system  $\tilde{H}^{n-1}(u)$  inducing  $\tilde{\gamma}$  on  $\tilde{S}^*_i$  which is not extendable at the origin. Then there exists $\Omega\in \Gamma(X, \Omega^n_X)$   and $l'\in[1,k] $  with the following properties:
\begin{enumerate}
\item  $\displaystyle \d\Omega = (m+u)\frac{\d f}{f} \wedge \Omega + \frac{l'}{k}\frac{\d t}{t} \wedge \Omega $;
\item along  $S^*$   the $n-$meromorphic $(\delta_u-\frac{l'}{k}\frac{dt}{t}\wedge)-$closed  form $\Omega/f^{m}$ induces  $\check{\tau}_1(\sigma) = \frac{\d f}{f}\wedge \sigma$ in the sheaf $h^n(\Gamma_{l'})$, for some global section  $\sigma$   on  $S^*$  of the sheaf  $h^{n-1}(\Gamma_{l'})$;
\item  the current on  $X$   of type $(n+1,0)$ with support  $\{0\}$ :
$$  P_2\big( \lambda = -m-u,Pf(\mu = -(k-l')/k, \int_X \ \vert f\vert^{2\lambda}\bar{f}^{m-j}
      \vert t \vert^{2\mu}\frac{\d f}{f}\wedge \Omega \wedge \square )\big) $$
      defines a non zero class in \ $H^{n+1}_{[0]}(X, \mathcal{O}_X)$ \ for  $j $  large enough
      in $\mathbb{N}$. 
\end{enumerate}
\end{Thm} 

\bigskip

\begin{Rem}
As a consequence, with the aid of Corollary \ref{Cor:oblic2}, we get an oblique pole of  $\int_X|f|^{2\lambda}|g|^{2\mu}\square$ through $(-m-u-j, -l'/k)$  for  $j \gg 1$, provided
$\int_X |f|^{2\lambda}\square$ does not have double pole at $-u-m$, for all $m\in \NN$.
\end{Rem}

\bigskip

\begin{proof}
As our assumption implies that  $H^1_0(\tilde S, \tilde  H^{n-1}(u))\neq 0$,
Proposition 10 and Theorem 13  of \cite{B:91} imply the existence of $\tilde\Omega\in \Gamma(\tilde X, \Omega^n_{\tilde X})$ verifying
\begin{enumerate}[(i)]
\item  $\displaystyle \d\tilde \Omega = (m+u)\frac{\d\tilde f}{f} \wedge \tilde\Omega  $, for some $m\in \NN$;
\item $\tilde{ob}_1([\tilde\Omega]) \neq 0$, that is  $\tilde{\Omega}/\tilde{f}^m$  induces, via the isomorphisms (\ref{E:isoms}),  an element in  $H^0(\tilde{S}^*, \tilde{H}^{n-1}(u))$  which is not extendable at the origin;
\item $Z\tilde\Omega=\zeta^{-l'}\Omega$, for some $l'\in[1,k]$ and $\zeta:=\exp(-2i\pi/k)$.
\end{enumerate}

\smallskip

Define then $\gamma'\in\H_0$ by the following condition: $(\pi_*)^{-1}\gamma'$ is the value at $\tau_0$ of  $\tildeob_1([\tilde\Omega])$. After condition (iii) we have
$\Theta_0(\gamma')=\zeta^{-l'}\gamma'$.

As we did in (\ref{E:decomp}), we may write
$$  \tilde\Omega = \sum_{\ell=0}^{k-1} \tau^{\ell-k} p^*\Omega _\ell.
$$
Put $\Omega:=\Omega_{k-l'}$. Because $p^*\Omega =\tau^{k-l'}\tilde\Omega_{k-l'}$ and $\tilde\Omega_\ell$ satisties
$\d\tilde \Omega _\ell= (m+u)\frac{\d\tilde f}{f} \wedge \tilde\Omega_\ell  $ for any $\ell$, property (1) of the Theorem is satisfied thanks to  injectivity of $p^*$.

Relation (iii) implies that $\tilde\Omega_{k-l'}$ induces $\tilde\gamma':=(\pi_*)^{-1}\gamma'$
and $\tilde\Omega_\ell$ induces 0 for $\ell\neq k-l'$. Hence condition (2) of the Theorem is satisfied.

In order to check condition (3), observe that the image of $r_j(\tilde\gamma')$  in $H^{n+1}_{[0]}(\tilde X, \Omega^{n+1}_{\tilde X})$ is equal to the conjugate of
$$ \d'\Res(\lambda=-m-u, \int_{\tilde X}|\tilde f|^{2\lambda}\bar{\tilde f}^{-j}\tilde \Omega_{k-l'}	\wedge\square)  = 	P_2(\lambda=-m-u, \int_{\tilde X}|\tilde f|^{2\lambda}\bar{\tilde f}^{-j} \frac{\d \tilde f}{\tilde f}    	\wedge\tilde\Omega_{k-l'}	\wedge\square ).
$$
After \cite{B:91}, this current is an analytic nonzero functional supported in the origin in $\tilde X$. There exists therefore $\tilde w\in\Gamma(\tilde X, \tilde\Omega^{n+1}_{\tilde X})$ such that
$$ P_2(\lambda=-m-u, \int_{\tilde X}|\tilde f|^{2\lambda}\bar{\tilde f}^{-j} \frac{\d\tilde f}{\tilde f}    	\wedge\tilde\Omega_{k-l'}	\wedge\chi \bar{\tilde w}) \neq 0,
$$
for any cutoff $\chi$ equal to 1 near 0. The change of variable $\tau\mapsto\zeta\tau$ shows that $\tilde w$ may be replaced by its component $\tilde w_{k-l'}$ in the above relation. With $w\in
\Gamma (X, \Omega^{n+1}_{ X})$ such that $p^*w= \tau^{k-l'}\tilde w_{k-l'}$ we get
$$ P_2(\lambda=-m-u, \int_{ X}| f|^{2\lambda}\bar{ f}^{-j}|t|^{-2(k-l')/k} \frac{\d f}{ f}    	\wedge\Omega_{k-l'}	\wedge\chi \bar{ w}) \neq 0.
$$
\end{proof}

\bigskip

\begin{Rem}
The case $l'=k$ is excluded if $\int_X | f|^{2\lambda}\square$ has only simple poles at
$-m-u$  for all  $m \in \mathbb{N}$. Indeed, if $l'=k$, the class $[\Omega]$ in $H^n(u)$ satisfies $\ob_1([\Omega])\neq 0$; from Theorem 13 of \cite{B:91}  interaction of strata is present and gives rise to  poles of order  $\geq 2$.
\end{Rem}

\bigskip

\section{Examples}

\begin{Ex} \label{Ex:premier}
$n=2$, $f(x,y,t)=tx^2-y^3$.
The extension of  $\int_X|f|^{2\lambda}|t|^{2\mu}\square$ presents an oblique polar line of direction $(3,1)$ through $(-5/6-j, -1/2)$, for $j\gg 1$. In fact it follows from general facts that $j=2$ is large enough because here $X$ is a neighborhood of $0$ in $\CC^3$.
\end{Ex}
\begin{proof}
We verify directly that the standard generator of $H^1(5/6)$ (which is a local system of rank 1)
on $S^*:= \ens{ x=y=0} \cap \ens{t\ne 0}$ has monodromy $-1 = \exp(2i\pi 1/2 )$.
We take therefore $k=2$ and we have $\tilde f(x,y,\tau):= \tau^2 x^2-y^3$.

Put 
$$ \tilde S^*=\tilde S_1^* \cup \tilde S_2^* \ \text{with} \ 
     \tilde S_1^*:=  \ens{ x=y=0} \cap \ens{\tau\ne 0}, \ 
     \tilde S_2^*:=  \ens{ \tau=y=0} \cap \ens{x\ne 0}.
$$     

The form $\tilde \omega :=  3x\tau \, dy -2y \d(x\tau) $ verifies
\begin{equation} \label{E:domega}
          \d\tilde\omega =\frac{5}{6}\frac{\d \tilde f}{\tilde f}\wedge\tilde\omega
\end{equation}
and $\tilde\omega$ induces a nonzero element in the $H^1$ of the Milnor fibre of $\tilde f$ at $0$
because it induces on $\tilde S_1^*$ the pullback of the multivalued section of $H^1(5/6)$ we started with. It follows that the form $\omega$ of Theorem \ref{Thm:noninter} is
$$ \omega = 3xt\,dy-2yt\,dx -xy \,dt.
$$
It verifies $p^*\omega = \tau \tilde \omega$ and hence
$$ \d \omega = \frac{5}{6}\frac{\d f}{f}\wedge \omega + \frac{1}{2}\frac{d t}{t}\wedge \omega.
$$
One way to see interaction of strata for $\tilde f$ and $\exp(2i\pi 5/6)$ consists in looking at the form 
$\tilde\Omega:=\frac{d\tau}{\tau}\wedge\tilde\omega = d\tau\wedge(3x\,dy-2y\,dx)$ that verifies
$\d\tilde\Omega =  \frac{5}{6}\frac{\d\tilde f}{\tilde f}\wedge \tilde \Omega$.
Along $\tilde S_1^*$ we have 
$$ \tilde \Omega = \d(\tilde\omega\log\tau)-\frac{5}{6}\frac{\d\tilde f}{\tilde f}\wedge \tilde \omega\log      \tau,
$$
after (\ref{E:domega}).
Hence 
$\ob_1(\tilde \Omega)\ne 0$ in $ H^1(S_1^*,\tilde H^1(5/6))$. Interaction of strata is proved.

\smallskip
It turns out that Theorem \ref{Thm:inter} may also be used to see existence of an oblique pole as follows. Construct a section on
$\tilde S^*$ of $\tilde H^1(5/6)$ that does not extend through $0$ by setting $0$ on $\tilde S_2^*$
and the restricition of $\tilde \omega$ to $\tilde S_1^*$. This section does not extend because  otherwise its value at the origin  should be  not $0$ in $H^1$ of $\tilde f$ (because not $0$ along
$\tilde S_1^*$) on one hand and should be $0$, because of its value on $\tilde S_2^*$, on the other hand.

Notice also that the meromorphic extension of $\int_X |f|^{2\lambda}\square$ does not have a double pole at $-5/6-j$, for all $j\in\NN$, because interaction of strata is not present for $\exp(-2i\pi 5/6)$: the monodromies for
$H^1$ and $H^2$ of the Milnor fibre of $f$ do not have the eigenvalue $\exp(-2i\pi 5/6)$ because they are of order 3 thanks to homogeneity.
\end{proof}

\begin{Ex}
$n=2$, $f(x,y,t)=x^4+y^4+tx^2y$.
The extension of  $\int_X|f|^{2\lambda}|t|^{2\mu}\square$ presents an oblique polar line of direction $(4,1)$ through $(-5/8, -1/2)$.
\end{Ex}

\begin{proof}
The Jacobian ideal of $f$ relative to $t$, denoted by  $J_/(f)$, is generated by
\begin{equation*}  
\frac{\partial f}{\partial x} = 4x^3 + 2txy \ \ \text{and}  \ \ \frac{\partial f}{\partial y} = 4y^3 + tx^2.
\end{equation*}
We have
\begin{equation}  \label{E:2}
t\frac{\partial f}{\partial x} - 4x\frac{\partial f}{\partial y} = 2(t^2 - 8y^2)xy .
\end{equation}
Put $\delta : = t^2 - 8y^2$ and notice that for  $t \not=0$ the function $\delta$  is invertible
at $(t, 0, 0)$.
We use notations and results of \cite{B:08}. Recall that
$\mathbb{E} : = \Omega^2_/\big/ \d_/f\wedge \d_/\mathcal{O}$
is equipped with two operations $a$ and $b$ defined by $a\xi =\xi f$, $b(\d_/\xi) : = \d_/f\wedge \xi$ and a $t-$connection $b^{-1}.\nabla : \mathbb{P} \to \mathbb{E}$ that commutes to $a$ and $b$ where
\begin{enumerate}
\item \ $\nabla : \mathbb{E} \to \mathbb{E} $  is given by $\nabla(\d_/\xi) : = \d_/f\wedge \frac{\partial \xi}{\partial t} - \frac{\partial f}{\partial t}d\xi $,
\item  \ $\mathbb{P} : = \{ \alpha \in \mathbb{E} \mid \nabla(\alpha) \in b\mathbb{E}\}$.
\end{enumerate}
Relation (\ref{E:2}) gives
\begin{equation} \label{E:3}
2xy\delta = \d_/f \wedge (t\,dy + 4x\,dx) = \d_/f \wedge \d_/(ty + 2x^2);
\end{equation}
hence $xy\delta =0 \in \mathbb E$, and $xy \in J_/(f)$ for $t\ne 0$. As a consequence,
for $t \not= 0 $ fixed,  $x^3$ and $y^4$ belong to  $J(f_t)$. Therefore the $ (a,b)-$module
$\mathbb{E}_{t_0} : = \mathbb{E}/(t-t_0)\mathbb{E}$ has rank  $5$ over $\mathbb{C}[[b]]$.
The elements  $1, x, y, x^2, y^2 $ form a basis of this module.

We compute now the structure of $(a,b)-$module of $\mathbb E$ over the open set
$ \{t \not= 0\}$, \emph{i.e.}, compute the action of $a$ on the basis. Let us start with
$$ a(y^2) = x^4y^2 + y^6 + x^2y^3 .$$
Relation (\ref{E:2}) yields
\begin{equation} \label{E:4}
2x^2y\delta = \d_/f \wedge (tx\,dy + 4x^2\,dx) = t.b(1)
\end{equation}
and also
 \begin{equation} \label{E:5}
  x^2y   = b(\d(\frac{tx\,dy + 4x^2\,dx}{2\delta})) =  \frac{1}{2t}b(1) + \frac{4}{t^3}b(y^2) + b^2.\mathbb{E} .
 \end{equation}
 From
 $$b(1) = \d_/f \wedge (x\,dy) = 4x^4 + 2tx^2y = \d_/f \wedge (-y\,dx) = 4y^4 + tx^2y
 $$
 we get
 \begin{equation}\label{E:6}
 4x^4 = -\frac{8}{t^2}b(y^2) + b^2\mathbb{E} .
 \end{equation}
 Therefore
 \begin{equation} \label{E:7}
 4y^4 = b(1) - tx^2y = \frac{1}{2}b(1) - \frac{4}{t^2}b(y^2) + b^2\mathbb{E} .
 \end{equation}
 The relation
  $$ x^4y^2 = \d_/f \wedge \frac{x^3y\,dy + 4x^4y\,dx}{2\delta}$$
  deduced from (\ref{E:3}) shows $x^4y^2 \in b^2 \mathbb{E}$.

  Relation (\ref{E:4}) rewritten as  $2t^2x^2y = tb(1) + 16x^2y^3$ yields
  \begin{equation*} 
  x^2y^3 = \frac{t^2}{8}x^2y - \frac{t}{16}b(1).
  \end{equation*}
Moreover
  \begin{equation*} 
  y^3(4y^3 + t x^2) = y^3\frac{\partial f}{\partial y} = \d_/f \wedge (-y^3\,dx) = 3b(y^2)
  \end{equation*}
  and hence
  \begin{equation*} 
  4y^6 = -tx^2y^3 + 3b(y^2).
  \end{equation*}
 On the other hand
  \begin{equation*} 
  b(y^2) = \d_/f \wedge (xy^2\,dy) = 4x^4y^2 + 2tx^2y^3 = 2tx^2y^3 + b^2\mathbb{E}.
  \end{equation*}
Finally
\begin{equation*}  
	a(y^2)  =   x^4y^2 + y^6 + tx^2y^3 =
                    -\frac{t}{4}x^2y^3 + \frac{3}{4}b(y^2) + tx^2y^3 +  b^2\mathbb{E}
    		 =   \frac{9}{8}b(y^2) +  b^2\mathbb{E}.
\end{equation*}

Now, after (\ref{E:6}), (\ref{E:7}) and (\ref{E:5}) we obtain successively
$$ \begin{aligned}
    a(1) &= x^4 + y^4 + tx^2y = -\frac{2}{t^2}b(y^2) + \frac{1}{8}b(1) - \frac{1}{t^2}b(y^2) +  	\frac{1}{2}b(1) + \frac{4}{t^2}b(y^2) +  b^2\mathbb{E} \\
           &= \frac{5}{8}b(1) + \frac{1}{t^2}b(y^2) + b^2\mathbb{E}, \\
        a(1 - \frac{2y^2}{t^2}) &= \frac{5}{8}b(1) + \frac{1}{t^2}b(y^2) - \frac{2}{t^2} 	\frac{9}{8}b(y^2) + b^2\mathbb{E} =\frac{5}{8}b(1 - \frac{2y^2}{t^2}) + b^2\mathbb{E} .
	 \end{aligned}
$$
Some more computations of the same type left to the reader give

$$   \begin{aligned}
    a(x) &= b(x) + b^2\mathbb{E} ,\\
    a(y) &= \frac{7}{8}b(y) + b^2\mathbb{E} , \\
    a(x^2) &= \frac{11}{8}b(x^2) + b^2\mathbb{E} .
	\end{aligned}
$$

   \bigskip

Let us compute the monodromy $M$ of $t$ on the eigenvector $v_0:=1 - \frac{2y^2}{t^2} +b\mathbb{E} $. Because   $b^{-1}\nabla = \frac{\partial}{\partial t}$ it is given by
$$ M =  \exp(2i\pi tb^{-1}\nabla).
$$
We have
$$
   \nabla(1) = - x^2y   = -b(\frac{1}{2t}1 + \frac{4y^2}{t^3}) + b^2\mathbb{E}
$$
and hence
$$
   t\frac{\partial}{\partial t}(1) = \frac{-1}{2}1 - \frac{4y^2}{t^2} + b\mathbb{E}.
$$
Also
$$
    \nabla(y^2) = -x^2y^3 = -\frac{t^2}{8}x^2y + \frac{t}{16}b(1) \\
    	  =  -\frac{t^2}{8}\big(\frac{1}{2t}b(1) + \frac{4}{t^3}b(y^2) \big) + \frac{t}{16}b(1) + b^2\mathbb{E}
$$
gives
$$
    t\frac{\partial}{\partial t}(y^2) = -\frac{1}{2}y^2 +  b\mathbb{E} .
$$
Hence
$$  t\frac{\partial}{\partial t}(1 - \frac{2y^2}{t^2}) =
    -\frac{1}{2}(1 - \frac{2y^2}{t^2}) + b\mathbb{E}
$$
from what we deduce $Mv =-v$.

         \bigskip
An analogous computation with  $\tau^2 = t$, shows that the eigenvector $\tilde{v}$ is invariant under $\tilde{M}$. On the other hand, the relation
$$ \tilde{v} = 1 - \frac{2y^2}{\tau^4} + b\tilde{\mathbb{E}} $$
where $\tilde{\mathbb{E}} $ is associated to the pair $(\tilde{f},\tau)$, with
$ \tilde{f}(x,y,\tau) : = x^4 + y^4 + \tau^2x^2y $,
shows that $\tilde{v}$ does not extend through $0$ as a section of  $\tilde{\mathbb{E}}$.

    \bigskip

This last assertion may be proved directly. It suffices to show that there does not exist a holomorphic non-trival\footnote{that is, not inducing $0$ in the Milnor fibre of $\tilde f$ at $0$}
$1-$form $\tilde \omega$ near $0$  such that
\begin{equation} \label{E:13}
   \d\tilde{\omega} = \frac{5}{8}\frac{\d\tilde{f}}{\tilde{f}}\wedge \tilde{\omega}.
\end{equation}
Because $\tilde f$ is quasi-homogeneous  of degree $8$ with the weights $(2,2,1)$ and
because $\tilde \omega/\tilde f^{5/8}$ is homogeneous of degree $0$, the form
$\tilde\omega$ must be homogeneous of degree $5$. So we may write
$$\tilde{\omega} = (\alpha_0 + \alpha_1\tau^2 + \alpha_2\tau^4)d\tau + \tau\beta_0 + \tau^3\beta_1
$$
where $\alpha_i$ and $\beta_i$ are respectively $0-$ and $1-$homogeneous forms of degree $2 -i$ with respect to $x,y$. Setting $\beta : = \beta_0 + \tau^2\beta_1$,
we get
$$
     \d\tilde{f} \wedge \tilde{\omega} = \d_/\tilde{f}\wedge \tau\beta  \quad \text{and} \quad
    \d\tilde{\omega} = \tau d_/\beta \quad  \text{modulo} \quad d\tau\wedge \square.
$$
With (\ref{E:13}) we deduce
$$ 8\tilde{f}\d_/\beta = 5\d_/\tilde{f} \wedge \beta
$$
 and an easy computation shows that this can hold only if $\beta = 0$. In that case $\alpha = 0$ also and the assertion follows.
\end{proof}

\bibliographystyle{amsplain}
\bibliography{BarMai09}

\end{document}